\documentclass{article}

\usepackage{mathtools}
\usepackage{amsthm}
\usepackage{amssymb}
\usepackage{dsfont}
\usepackage{nicefrac}
\usepackage{todonotes}
\usepackage{xcolor}
\newcommand{\R}{\mathds{R}}
\newcommand{\co}{\operatorname{co}}

\newcommand{\F}[1]{\mathcal{F}(#1)}
\newcommand{\Cf}{\mathcal{C}}
\newcommand{\supp}[2]{\operatorname{supp}(#2,#1)}
\newcommand{\argmin}{\mathop{\mathrm{argmin}}}

\newcommand{\mto}{\rightrightarrows}
\usepackage{pgfplots}
\usepackage[hidelinks]{hyperref}
\usepackage{caption}
\usepackage{subcaption}
\usepackage{enumitem}

\usepackage[backend=biber,style=numeric-comp,doi=false,isbn=false,url=false,giveninits]{biblatex}
\bibliography{references.bib}

\newcommand{\epi}{\operatorname{epi}}
\newcommand{\dom}{\operatorname{dom}}
 
\newcommand{\im}{\operatorname{Im}}
\newcommand{\ri}{\operatorname{ri}} 
\newcommand{\abs}[1]{\lvert #1 \rvert}
\pgfplotsset{compat=1.17}
\newtheorem{theorem}{Theorem}
\newtheorem{corollary}{Corollary}
\newtheorem{proposition}{Proposition}
\theoremstyle{plain}
\newtheorem{definition}{Definition}
\newtheorem{lemma}{Lemma}
\newtheorem{assumption}{Assumption}
\newtheorem{problem}{Open Problem}
\newtheorem{example}{Example}
\theoremstyle{remark}
\newtheorem{remark}{Remark}

\newcommand{\email}[1]{\href{mailto:#1}{Email: #1}}

\title{Variational properties of the abstract subdifferential operator}

\author{R. D\'iaz Mill\'an \thanks{Deakin University, 75 Pigdons Rd, Waurn Pond, VIC 3216
  (\email{r.diazmillan@deakin.edu.au})} \and N. Sukhorukova \thanks{Swinburne University of Technology, PO Box 218, Hawthorn, VIC 3122
  (\email{nsukhorukova@swin.edu.au})} \and J. Ugon \thanks{Deakin University, 75 Pigdons Rd,
 Waurn Pond, VIC 3216 (\email{julien.ugon@deakin.edu.au})}}
 \date{}

\begin{document}
  \maketitle
  \begin{abstract}
    Abstract convexity generalises classical convexity by considering the suprema of functions taken from an arbitrarily defined set of functions. These are called the abstract linear (abstract affine) functions. The purpose of this paper is to study the abstract subdifferential. We obtain a number of results on the calculus of this subdifferential: summation and composition rules, and prove that under some reasonable conditions, the subdifferential is a maximal abstract monotone operator.
    Another contribution of this paper is a counterexample that demonstrates that the separation theorem between two abstract convex sets is generally not true. The lack of the extension of separation results to the case of abstract convexity is one of the obstacles in the development of numerical methods based on abstract convexity.
  \end{abstract}
  
  \textbf{Keywords:} {abstract convex functions, abstract subdifferential, abstract monotonicity}
  
\textbf{MSC2020:}{ 52A01, 90C30,  47N10,  49J52,  49J53}

  \section{Introduction}

  Given a set $L$ of functions (also known as elementary functions), a function is $L$-convex, or \emph{abstract convex} when it can be expressed as the supremum of functions from $L$. Abstract convexity (also known as convexity without linearity) has been extensively studied since the 1970s, and is the subject of at least three monographs~\parencite{FoundationsOfPallas1997,rubinov:2000,Singer}, and many papers~\parencite{andramonov:2002,bui.ea:2021,burachik.ea:2004,BurachikRubinov,diaz-millan.ea:2023,kutateladze72} just to name a few. Many notions and results from classical convexity have been generalised, including duality~\parencite{kutateladze72}, the conjugate function~\parencite{rubinov:2000}, the subdifferential~\parencite{rubinov:2000} and normal cones~\parencite{jeyakumar.ea:2007}. This paper aims to investigate abstract subdifferential calculus, which will be a significant step towards the development of numerical methods based on abstract convexity. 

  In convex analysis, the subdifferential plays an essential role in developing algorithms. A major challenge in generalising the properties of the convex subdifferential to the abstract convex subdifferential is that most proofs rely on Minkowski's theorem on the separation of sets. However, such a theorem is not known in abstract convex analysis, and in fact, we show in this paper that its generalisation is not true in general. More specifically, there might not exist a function $l\in L$ separating two $L$-convex sets. Consequently, we prove the results in this paper without resorting to such separation results. 

  Another challenge in the application of abstract convexity to practical problems is the choice of an efficient set of elementary  functions. In some cases, for example uniform approximation problems, the choice is straightforward, but this situation is rather exceptional. For some classes of functions, such as quasiconvex functions, some elementary sets of functions are known~\cite{rubinov:2000,daniilidis.ea:2002} but counter-examples presented in~\cite{burachik.ea.08} and in this paper show that they do not always enable the generalisation of results such as the maximality of the subdifferential operator, or the maximum rule. 

  The study of the abstract subdifferential is the subject of a few papers, and several results, such as the sum rule or the maximal monotonicity of the subdifferential operator~\cite{jeyakumar.ea:2007,BurachikRubinov,doagooei.ea:2013} were generalised under some assumptions. Nearly all of these papers make the assumption that the set $L$ is closed under the sum (i.e., $L$ is additive). This is a strong assumption, which excludes many families of functions of practical interest, such as quasilinear and quasiconvex functions. It is also generally assumed that $0\in L$, which is a more reasonable assumption, as this it is required to verify optimality conditions. Finally, several papers require that abstract convex functions are additive, that is, they only consider functions $f$ such that $f(x)+f(y)=f(x+y)$, also adding further hypotheses on the domain $X$ of these functions. In most of the present paper we make no such assumption on $X$ and $L$. In fact we derive a number of subdifferentiable calculus rules under very general conditions, in such a way that these rules generalise fairly precisely subdifferential calculus, both in the convex case and the known abstract convex cases. Only in the last part of the paper, when we investigate the maximal monotonicity of the subdifferential operator, do we introduce a few assumptions on the set $X$. 

  The paper is divided into three parts. First, in Section~\ref{sec:preliminaries} we recall some important results and definitions of abstract convexity and clarify the notations used throughout the paper. Then, in Section~\ref{sec:calculus}, we generalise several well-known properties, including the sum rule on the subdifferential of the sum of abstract convex functions and the composition rule. Finally, in Section~\ref{sec:monotoneoperator}, we prove that under suitable conditions the abstract subdifferential is a maximal abstract monotone operator. 
  %In the same section, we develop subdifferential calculus results in the case of abstract convexity: summation and compositions rules.
  Section~\ref{sec:conclusions} contains the conclusions and open problems.

  \section{Preliminaries}\label{sec:preliminaries}
 
   We start by recalling the main definitions and properties of abstract convexity. We will use the notations introduced by~\textcite{rubinov:2000}.

  \subsection{Definitions}

  Let $X$ be a set and $\F{X}$ be the set of all real-valued functions with domain on $X$ as well as the function identically equal to $-\infty$:
  \[\F{X} = \{f:X\to \R\cup\{+\infty\}\}\cup \{-\infty:X\to \R, -\infty(x)=-\infty\}.\] Let $H\subset \F{X}$ be an arbitrary set of functions.

  \begin{definition}[Abstract Convexity~\parencite{rubinov:2000}]
    A function $f\in \F{X}$ is said to be \emph{($(H,X)$-convex} (or simply \emph{$H$-convex} when the context is clear) if there exists a set $U\subset H$ such that for any $x\in X$, $f(x) = \sup_{u\in U} u(x)$.
  \end{definition}

  \begin{definition}[Support Set~\parencite{rubinov:2000}]
  The \emph{$H$-support set} of a function $f\in \F{X}$ is defined as:
  \[ \supp{H}{f} = \{u\in H: u\leq f\} \]
  where $u\leq f$ means that $u(x)\leq f(x), \forall x\in X$.
  \end{definition}

  The $H$-support sets of $H$-convex functions are called \emph{$(H,X)$-convex sets}, or \emph{$H$-convex sets} when the context is clear. Note that the $H$-support set of every function $f\in\F{X}$ is $H$-convex, and there is a one-to-one mapping between $H$-convex functions and $H$-convex sets.

  The intersection of $H$-convex sets is $H$-convex:
  \begin{proposition}[\parencite{diaz-millan.ea:2023}]\label{prop:intersection}
    Let $A_i: i\in I$ be a family of $H$-convex sets. Then $\cap_{i\in I}A_i$ is also $H$-convex.
  \end{proposition}

  For a set $L\subset \F{X}$, we denote by $H_L$ its vertical closure: $H_L = \{h_{lc} = l-c: l\in L, c\in \R\}$. The functions in $H_L$ are called \emph{L-affine} or \emph{abstract affine}. % The set $L$ is called a set of \emph{abstract linear functions} if $h_{l,c}\notin L$ for any $l\in L$ and $c\neq 0$.
  It can also be necessary to consider the vertical shift $c\in \R$ of an affine function $l-c$ as a variable. In this case, we denote the set of such functions by $L\times \{-1\}$ and its domain $X\times \R$, where $-1(c)\coloneqq -c$, and $(l,-1)\in L\times \{-1\}$ is defined as $(l,-1)(x,c) \coloneqq l(x)-c$.

    \begin{definition}[Abstract Convex Hull of a function~\parencite{rubinov:2000}]\label{def:convexhullfunction}
      Let $f\in\F{X}$. The function $\co_H f$ defined by
      \[
      \co_H f(x) \coloneqq \sup\{h(x): h\in \supp{H}{f}\}
      \]
      is called the \emph{$H$-convex hull} of the function $f$.
    \end{definition}
    It is easy to see that $\co_H f$ is the $H$-convex lower envelope of the function $f$, that is, the supremum of all $H$-convex functions that lie below $f$. In particular, $f$ is $H$-convex if and only if $f=\co_H f$. For a given set $L$ of abstract linear functions, we will denote by $L_x$ the set of abstract affine functions that all vanish at $x$: \begin{equation}
      L_x = \{l_x\coloneqq l-l(x): l \in L\}. \label{eq:Lx}
    \end{equation}
    The set $L_x$ has the same vertical closure as $L$, and its relevance will be clear when we discuss the abstract subdifferential.
    
  \begin{definition}[Abstract Convex Hull of a set~\parencite{rubinov:2000}]\label{def:convexhullset}
    The intersection of all $H$-convex sets containing a set $C\subset H$ is called the \emph{abstract convex hull} of $C$ and denoted $\co_H C$. 
  \end{definition}

  \begin{proposition}[{\parencite[Proposition 1.1 and Corollary 1.1]{rubinov:2000}}]\label{prop:convexhullsup}
    The $H$-convex hull of a set $C$ is $H$-convex. More specifically, it is the smallest $H$-convex set containing $C$, and $\co_H C = \supp{H}{\sup_{l\in C} l}$.
  \end{proposition}

%  \begin{proof}
%    For each $i\in I$, let $f_i(x) = \sup_{l\in A_i} l(x)$ and consider the function $f(x) = \inf_{i\in I}f_i(x)$. We will show that $\cap_{i\in I} A_i = \supp f$, which implies that $\cap_{i\in I} A_i$ is $L$-convex (note that $\supp f$ is the support set of the lower $L$-convex envelope of $f$).
%
%    \begin{align*}
%      \supp f &= \{l\in L: l(x) \leq \inf_{i\in I}f_i(x) \forall x\in X \} \\&= \{l\in L: l(x)\leq f_i(x) \forall i\in I, x\in X\} \\&= \cap_{i\in I}\{l\in L: l(x)\leq f_i(x)\forall x\in X\} = \cap_{i\in I} A_i
%    \end{align*}
%  \end{proof}

  Each abstract convex concept has a dual counterpart, which can be obtained by considering a coupling function \(\langle\cdot,\cdot\rangle: L\times X \to \R\), \(\langle u,x\rangle\coloneqq u(x)\). This coupling function enables one to obtain dual versions of the properties of $(L,X)$-convex functions and $(L,X)$-convex sets.

  \begin{definition}[$(X,L)$-convex set]
    Let $L\subset \F{X}$ be the set of abstract linear functions. We say that the set $C\subset X$ is a $(X,L)-$convex set, if there exists a function $\lambda:L\rightarrow (\R\cup \{-\infty,+\infty\})$ such that $C=\{x\in X: u(x)\leq \lambda(u), \forall u\in L\}$.
  \end{definition}

  \begin{definition}[$(X,L)$-convex function]
    For $L\subset \F{X}$, the function $\lambda: L\to (\R\cup \{-\infty,+\infty\})$ is said to be $(X,L)$-convex if there exists $C\subset X$ such that for any $u\in L$, $\lambda(u) = \sup_{x\in C}u(x)$.
  \end{definition}

  It is clear that the set $C$ in the above definition is $(X,L)$-convex.

  \begin{definition}[Support set, dual version]\label{def:support_dual} For $L\subset \F{X}$, the support set of a function $\lambda: 
L\to (\R\cup \{-\infty,+\infty\})$ is defined as $\supp{X}{\lambda} \coloneqq \{x\in X: u(x)\leq \lambda(u), \forall u\in U\}$.

  \end{definition}

  Then it follows from the above definition and Proposition~\ref{prop:intersection} that
  \begin{proposition}
      Let $C_i:i\in I$ be a family of $(X,L)-$convex sets. Then $\cap_{i\in I}C_i$ is a $(X,L)-$convex set.
  \end{proposition}
  %\begin{proof}
  %    Notice that for all $i\in I$, there exist a set $U_i\subset L$ and a function $\lambda^i:L\rightarrow \R$ such that $C_i=\cap_{u\in U_i}\{x\in X: u(x)\leq \lambda^i(u)\}$. Now, considering the function $\lambda:\cup_{i\in I}U_i\subseteq L\rightarrow \R$ as $\lambda(u):=\sup_{i\in A(u)} \lambda^i(u)$, where $A(u):=\{i\in I: u\in U_i\}$ is the set of active index for the element $u\in L$ with respect of the sets $U_i, i\in I$, then 
  %    $$\bigcap_{i\in I}C_i=\bigcap_{u\in \cup_{i\in I} U_i}\{x\in X:u(x)\leq \lambda(u)\},$$
  %    proving that $\bigcap_{i\in I}C_i$ is an abstract convex set.
  %\end{proof}

  The set of all $(X,L)$-convex sets on the space $X$ (respectively $L$) is denoted by $\Cf (X,L)$ and by symmetry the set of all $(L,X)$-convex sets is denoted by $\Cf (L,X)$.
   \begin{definition}
       The abstract convex hull of a set $A\in X$ is defined as the intersection of all abstract convex sets containing $A$, denoted by $\co_X A:=\cap_{C\in \Cf (X,L)}\{C:A\subset C\}$.
   \end{definition}

  Separation results are essential in convex analysis.  The following proposition is straightforward (separation of a point from a set).

  \begin{proposition}[ Separating property~\cite{rubinov:2000}] Given an $H$-convex set $U$ and a point $l\in H\setminus U$ there exists $x\in X$: $l(x)>u(x)$ for any $u\in U$.
  \end{proposition}

  \begin{corollary}[Dual separating property]
    Given an $(X,L)$-convex set $C$ and a point $x\in X\setminus C$ there exists a function $u\in H$: $u(x)>u(y)$ for any $y\in C$.
  \end{corollary}

  The above corollary has a simple geometric interpretation: for any abstract convex set $C\subset X$ and any point $x$ outside of that set, one can find an abstract affine function $u\in H$ with a level set (ie, an abstract hyperplane) separating the point $x$ from the set $C$. One may wonder whether it is possible to separate two abstract convex sets using such an abstract hyperplane. We show that this is not the case in general. Let us first formalise the notion of set separation below.

 The next definition states the separation between two abstract convex sets. 
 \begin{definition}[Separation $(X,H)$-convex sets]
   Consider a set of functions $H\in \F{X}$, and the $(X,H)$-convex sets $A,B \subset X$. We say that $A$ is $H$-separable from $B$ if there exists $u\in H$ such that $u(x)\leq u(y)$ for all $x\in A$ and $y\in B$. The separation is strict of the inequality is strict. 
 \end{definition}
     
  \begin{definition}[Separation of $(H,X)$-convex sets]
   Consider a family of functions $H\in \F{X}$ and the $(X,H)$-convex sets $A,B \subset H$. We say that $A$ is $X$-separable from $B$ if there exists $x\in X$ such that $u(x)\leq l(x)$ for all $u\in A$ and $l\in B$. The separation is strict of the inequality is strict. 
 \end{definition}

  The following example illustrates that in general, it is not always possible to separate two disjoint $(X,H)$-convex sets.

  \begin{example}
    Let $X=\R$ and  $H = \{-\abs{x-1} +2, -\abs{x+1}+2, -\abs{x}+2,0\}$. Consider the $H$-convex sets $A=\{-\abs{x-1} +2, -\abs{x+1}+2\}$ and $B=\{ -\abs{x}+2,0\}$. Refer to Figure~\ref{fig:counterExampleSeparation}.

    It is easy to see that $A$ and $B$ are both $H$-convex, and disjoint. Indeed, let 
    $$f_A(x) =  \sup_{u\in A} u(x).$$ Then $$\{x: -\abs {x}+2 > u(x)~\forall u\in A\} = \{x: \abs{x} + 2 > f_A(x)\} = \left(-\frac{1}{2},\frac{1}{2}\right).$$ 
    Therefore $-\abs{x}+2$ is not in $\supp{H}{f_A} = \co_H A$. Likewise, 
    $$\{x: 0 > u(x)~\forall u\in A\} = (-\infty,-3)\cup (3,+\infty).$$ Therefore $0\notin \supp{H}{f_A}$. Yet, since $[(-\infty,-3)\cup (3,+\infty)]\cap \left(-\frac{1}{2},\frac{1}{2}\right)=\emptyset$, there is no point $x\in X$ that separates all functions in $B$ from $A$.

     On the other hand, let $f_B(x) = \sup_{v\in B} v(x)$ then $$\{-\abs{x\pm 1}+2> v(x)~ \forall v\in B\} = \{-\abs{x\pm 1}+2 > f_B(x)\} = \left(\mp 3,\pm \frac{1}{2}\right),$$ which shows that $A\cap \supp{H}{f_B} = \co_H B = \emptyset$. However, since $$\left(-3,-\frac{1}{2}\right) \cap \left(\frac{1}{2},3\right) = \emptyset,$$ there is no point $x\in X$ that separates all functions in $A$ from $B$.

    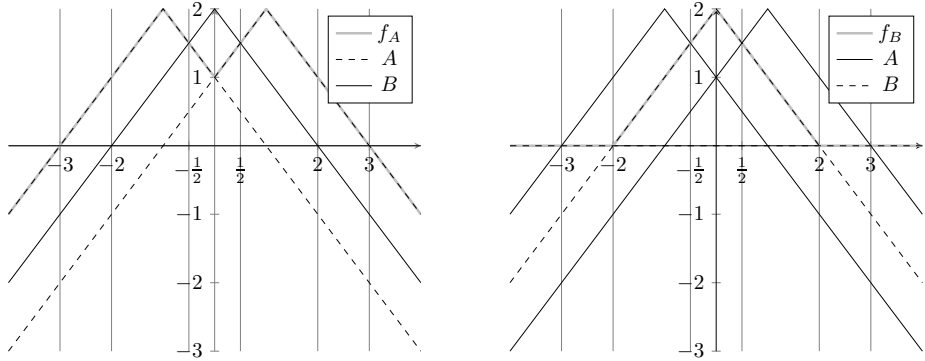
\begin{figure}
      \centering
      \begin{subfigure}[b]{0.45\textwidth}
       \begin{tikzpicture}[scale=0.8]
       \begin{axis}[axis lines = center, axis line style={thin, gray},xtick={-3,-2,-0.5,0.5,2,3},xticklabels={$-3$,$-2$,$-\frac{1}{2}$,$\frac{1}{2}$,$2$,$3$},grid style={line width=.1pt,draw=gray},xmajorgrids]
       \addplot[lightgray,very thick] coordinates {(-4,-1) (-1,2) (0,1) (1,2) (4,-1)}; \addlegendentry{$f_A$}
       \addplot[dashed] coordinates {(-4,-1) (-1,2) (4,-3)};\addlegendentry{$A$}
       \addplot[solid] coordinates {(-4,-2) (0,2) (4,-2)};\addlegendentry{$B$}
       \addplot[dashed] coordinates {(-4,-3) (1,2) (4,-1)};
       \addplot[solid] coordinates {(-4,0) (4,0)};
       \end{axis}
       \end{tikzpicture}
       \caption{Neither functions from $B$ are in the support set of the function $f_A$, and so $A$ is $H$-convex. Yet, there is no point $x\in \R$ at which both functions from $B$ are above $f_A$. Therefore the set $B$ cannot be separated from the set $A$.}
      \end{subfigure}\hfill
      \begin{subfigure}[b]{0.45\textwidth}
       \begin{tikzpicture}[scale=0.8]
         \begin{axis}[axis lines = center,xtick={-3,-2,-0.5,0.5,2,3},xticklabels={$-3$,$-2$,$-\frac{1}{2}$,$\frac{1}{2}$,$2$,$3$},grid style={line width=.1pt,draw=gray},xmajorgrids]
           \addplot[lightgray,very thick] coordinates {(-4,0) (-2,0) (0,2) (2,0) (4,0)}; \addlegendentry{$f_B$}
           \addplot[solid] coordinates {(-4,-1) (-1,2) (4,-3)};\addlegendentry{$A$}
           \addplot[dashed] coordinates {(-4,-2) (0,2) (4,-2)};\addlegendentry{$B$}
           \addplot[solid] coordinates {(-4,-3) (1,2) (4,-1)};%\addlegendentry{}
           \addplot[dashed] coordinates {(-4,0) (4,0)};
         \end{axis}
       \end{tikzpicture}
       \caption{Neither functions from $A$ are in the support set of the function $f_B$, and so $B$ is $H$-convex. Yet, there is no point $x\in \R$ at which both functions from $A$ are above $f_B$. Therefore the set $A$ cannot be separated from the set $B$.}
      \end{subfigure}
      \caption{The sets $A$ and $B$ are both $H$-convex and disjoint, but they cannot be separated by a point $x$.}\label{fig:counterExampleSeparation}
    \end{figure}
  \end{example}

  From here onward, we assume that $L\subset \F{X}$ is an arbitrary set of functions.

  \begin{definition}[Abstract subdifferential~\parencite{rubinov:2000}]
    The $L$-subdifferential of a function $f\in\F{X}$ at a point $x\in X$ is the set
    \[ \partial_L f(x) = \{l\in L: f(y)\geq f(x) + l(y) - l(x), \forall y\in X\}. \]
  \end{definition}

  We recall that the sublevel sets and the epigraph of a function $f$ are defined as follows:
  \begin{align*}
    S_c(f) &= \{x\in X: f(x)\leq c\}\\
    \epi(f) &= \{(x,c) \in X\times \R: f(x)\leq c\}
  \end{align*}
  When the domain of $f$ is ambiguous we denote the epigraph of $f$ over the domain $X$ as $\epi_X f$.

  Note that two functions are equal if and only if their epigraphs are equal.

  \begin{proposition}\label{prop:convexityepigraph}
  A function $f\in \F{X}$ is $L-$convex if and only if its epigraph is $(X\times \R,L\times \{-1\})$-convex.
  \end{proposition}
  \begin{proof}
    Let $f$ be any function and $u\in \supp{L}{f}$. Then,
    \begin{align*}
      &u\in \supp{L}{f} \\
      \iff& \forall y\in X, u(y)\leq f(y)\\
      \iff& \forall (y,c)\in X\times \R: c\geq f(y), u(y)\leq c\\
      \iff& \sup_{(y,c)\in \epi f} u(y)-c \leq 0\\
      \iff&\forall (y,c)\in \big(\co_{(X\times \R),(L\times \{-1\})} \epi f),  u(y)-c\leq 0 
    \end{align*}

    The function $f$ is $(L,X)$-convex if and only if for all $x\in X$, $d\in \R$ such that $d<f(x)$, there exists $u\in \supp{L}{f}$ such that $u(x)>d$. That is, for all $(x,d)\notin \epi f$,  $(x,d)\notin \co_{(X\times \R),(L\times \{-1\})} \epi f$.
    This proves the result.
  \end{proof}

  \begin{corollary}\label{cor:epigraphhull}
    For any function $f\in\F{X}$, $\co_{(X\times \R,L\times\{-1\})} \epi f = \epi \co_Lf$.
  \end{corollary}

  \begin{definition}[Abstract Normal Cone~\parencite{jeyakumar.ea:2007}]
    Consider a set $C\subset X$ and a point $x\in C$. The \emph{$L$-normal cone} (or \emph{abstract normal cone}) at $x$ in $C$, with respect to the set of functions $L$, denoted by $N_L(x,C)$ is defined as: 
    \[N_L(x,C):=\{l\in L: l(y)-l(x)\leq 0, \forall y\in C\}.\]
    By convention we say that $N_L(x,C) =\emptyset$ when $x\notin C$.
  \end{definition}
  %Considering $(x,\lambda)\in X\times \R$, $C\subset X\times \R$ note that $N_{L\times\R}((x,\lambda),C)=\{(l,\alpha)\in L\times \R: l(y)-l(x)+\alpha(\gamma-\lambda)\leq 0, \forall (y,\gamma)\in C\}$.

  \begin{proposition}\label{prop:convexityofsets}
    Suppose that the function $f\in\F{X}$ is $L$-convex. Then we have the following:
    \begin{enumerate}
      \item For any $c\in \R$, $S_c(f)$ is $(X,L)$-convex \parencite{kutateladze72,kutateladze76};
      \item For $C\subset X$, $N_{L_x}(x,C) = \{l - l(x): l\in N_L(x,C)\}$ is $(L_x,X)$-convex.%, where $L_x \coloneqq \{u - u(x): u\in L\}$.
    \end{enumerate}
  \end{proposition}
  \begin{proof}
    \begin{enumerate}
      \item Define the function $g:L\rightarrow \R$ as $g(l):= \sup_{x\in S_c(f)} l(x)$, and let $Y=\supp{X}{g}$. It is clear from the definition that $S_c(f)\subset Y$. Let $y\in Y$. Then, $l(y)\leq g(l)$ for any $l\in L$, and in particular for any $l\in \supp{L}{f}$. In other words,

      \[ l(y) \leq \sup_{x\in S_c(f)} l(x)\leq \sup_{x\in S_c(f)} f(x) \leq c \]
      Since the inequality $l(y)\leq c$ is true for any $l\in \supp{L}{f}$, it is true for the supremum of all functions in the set $\supp{L}{f}$, and so $f(y)\leq c$, which implies that $y\in S_c(f)$. Therefore $Y=S_c(f)$, which is an $(X,L)$-convex set.

      \item For $y\in C$, define $N_L(x,y) = \{l\in L: l(y) - l(x)\leq 0\} = \{l\in L_x: l(y)\leq 0$. Note that $N_L(x,y)$ is the sublevel set of the function $y:L_x\to \R$ at the level 0, and therefore it is a $L_x$-convex set.

      Since \[\{l - l(x): l\in N_L(x,C)\} = \bigcap_{y\in C} N_L(x,y)\] we can conclude that $N_L(x,C)$ is a $L_x$-convex set, by Proposition~\ref{prop:intersection}.\qedhere
    \end{enumerate}
  \end{proof}

  Note that $N_{L\times \{-1\}}((x,f(x)),\epi(f)) = \partial_L f(x) \times \{-1\}$.

  \begin{proposition}\label{normalsubdif}
  $\partial_L f(x)=\{l\in L: (l,-1)\in N_{L\times \R}\left((x,f(x)),\epi(f)\right)\}$.
  \end{proposition}
  \begin{proof}
  Consider $l\in \partial_L f(x)$, then for all $y\in X$, $f(y)\geq f(x)+ l(y)-l(x)$ implying that $l(y)-l(x)+ (-1)(f(y)-f(x))\leq 0$, then $(l,-1)\in N_L((x,f(x)),\epi(f))$.

  On the other hand, suppose that $(l,-1)\in N_L((x,f(x)),\epi(f))$, then for all $(y,\lambda)\in \epi(f)$ which means that $\lambda\geq f(y)$, we have $l(y)-l(x)+(-1)(\lambda-f(x))\leq 0$. In particular, since $(y,f(y))\in \epi(f)$, then $l(y)-l(x)+(-1)(f(y)-f(x))\leq 0$ implying that $l\in \partial_L f(x)$. 
  \end{proof}

  \subsection{Abstract Conjugate}

  Let us recall a few results on the abstract conjugate.
  \begin{definition}[Abstract conjugate \parencite{rubinov:2000}]
    The \emph{$L-$abstract conjugate} of the function $f\in\F{X}$, denoted by $f^*_L:L\rightarrow \R $ is defined by 
    \[f^*_L(u):=\sup_{x\in X}\{u(x)-f(x)\}.\]
  \end{definition}
  Since for all $x\in X$,  $f^*_L(u)\geq u(x)-f(x)$ we obtain the inequality \begin{equation}
  \label{eq:Moreau}
  f^*_L(u)+f(x)-u(x)\geq 0.
  \end{equation}

  \begin{proposition}[{\cite[Proposition 7.7]{rubinov:2000}}]\label{prop:moreau}
  Given $u\in L$ and $x\in X$, $f^*_L(u)+f(x)=u(x)$ if and only if $u\in \partial_L f(x)$.
  \end{proposition}

    \begin{proposition}[\cite{jeyakumar.ea:2007}]\label{prop:epiconjugate}
      Let $L$ be the set of abstract linear functions and $f\in\F{X}$ be any function. We have $\epi f^*_L = \supp{H_L}{f}$.
    \end{proposition}
%    \begin{proof}
%      We have:
%      \begin{align*}
%        \epi_L f^* &= \{ (l,c)\in L\times \R:c\geq f^*(l)\}\\
%        &= \{(l,c)\in L\times \R: c\geq l(x)-f(x) \forall x\in X\}\\
%        &= \{(l,c)\in L\times \R: l-c\leq f\}\\
%        &= \supp{H_L}{f}
%      \end{align*}
%    \end{proof}

%  \begin{proof}
%    According to Inequality~\eqref{eq:Moreau}, it suffices to show that $u\in \partial_L f(x)$ if and only if $f^*(u)+f(x)\leq u(x)$.
%    \begin{align*}
%      u(x) &\geq  f^*(u)+f(x)\\
%    u(x) &\geq \sup_{y\in X} (u(y) - f(y)) + f(x)\\
%    u(x) &\geq u(y) - f(y) + f(x) & \forall y\in X
%    \end{align*}
%  \end{proof}
%

\section{Abstract Subdifferential Calculus}\label{sec:calculus}

In this section, we develop a calculus for the abstract subdifferential, in particular, the sum and composition rules. 

  \begin{theorem}
    Let $L_1\in \F{X}$ and $L_2\in \F{X}$ be two families of abstract linear functions defined on $X$, such that $L_1\subset L_2$. Then for any $x\in X$, 
       \[ \partial_{L_2} f(x)\cap L_1 = \partial_{L_1} f(x)\] %\\
  \end{theorem}

  \begin{proof}
    Suppose that \(l\in L_1\cap \partial_{L_2}\partial f(x)\). Then for any $y\in X$, $f(y)\geq f(x) - l(x)+ l(y)$ and therefore $l\in \partial_{L_1}f(x)$.

    If $l\in \partial_{L_1}f(x)$, it follows that $l\in L_1$ and that $l\in \partial_{L_2}f(x)$.
  \end{proof}

\begin{lemma}
  Let $f$ be a $L$-convex function, and let $L_x$ be defined as in Equation~\eqref{eq:Lx}. Then, $\partial_{L_x} f(x)$  is a $L_x$-convex set.
\end{lemma}

\begin{proof}
  It suffices to prove that $\co_{L_x} \partial_{L_x} f(x) \subset \partial_{L_x} f(x)$, since the reverse inclusion is clear. Let $h = \sup_{l\in \partial_{L_x} f(x)} l$ and $u\in \co_{L_x} \partial_{L_x} f(x)$. Since $\partial_{L_x} f(x)\subset L_x$, the function $h$ is $L_x$-convex, and by definition of $\co_{L_x}$, the function $u$ is in the $L_x$-support set of $h$. Then for any $y\in X$,
  \begin{align*}
    u(y) - u(x) &\leq h(y) = \sup_{l\in \partial_{L_x} f(x)} l(y) \\
    &\leq  \sup_{l\in \partial_{L_x} f(x)} f(y) - f(x) + l(x) \\
    &= f(y) - f(x)
  \end{align*}
  from which we conclude that $u\in \partial_{L_x} f(x)$.
\end{proof}

\begin{remark}
$\partial_L f(x) = \{l\in L: l-l(x) \in \partial_{L_x} f(x)\}$. Indeed,
\begin{align*}
    \partial_L f(x) &= \{l\in L: \forall y\in X, f(y) \geq f(x) + l(y)-l(x)\}\\
   &= \{l\in L: \forall y\in X, f(y) \geq f(x) + l_x(y)-l_x(x)\}\\ 
   &= \{l\in L: l-l(x) \in \partial_{L_x} f(x)\}
\end{align*}
\end{remark}

\begin{remark}\label{rmk:diffconvexlinear}
  Let $f\in \F{X}$ be a function. We can develop simple rules to calculate the $L$-subdifferential of the function obtained from $f$ through horizontal or vertical shift:
  \begin{description}
  \item[Vertical shift:] For a function $u:X\rightarrow \R$, the $(L-u)$-subdifferential of the function $f_u = f-u$ at a point $x\in X$ is $\partial_{L-u} f_u(x) = \partial_L f(x) - u$. Furthermore $f_u$ is $(L-u)$-convex if and only if $f$ is $L$-convex.

    Indeed, \begin{align*}
    \partial_{L-u} f_u(x) &= \{l_u\in L-u: \forall z\in X f_u(z)\geq f_u(x) + l_u(z) - l_u(x)\} \\
                     &= \{\begin{multlined}[t]l-u: \forall z\in X, l\in L,\\f(z) - u(z)\geq f(x) -u(x) + l(z) - u(z) -l(x) + u(x) \}\end{multlined}\\
                     &= \{l-u:l\in\partial_L f(x)\}
    \end{align*}

  \item[Subdifferential vanishing at a point $y$:] For $y\in X$, the $L_y$-subdifferential of the function $f$ at a point $x\in X$ is $\partial_{L_y}f(x) = \{u-u(y): u\in \partial_L f(x)\}$.

    Indeed,
    \begin{align*}
    \partial_{L_y} f(x) &= \{l\in L_y: \forall z\in X: f(z)\geq f(x) + l(z) - l(x)\} \\
                       &= \{\begin{multlined}[t]l-l(y): l\in L, \forall z\in X:\\f(z)\geq f(x) + l(z) - l(y) - l(x) + l(y)\}\end{multlined} \\
                       &= \{l-l(y): l\in \partial_L f(x)\}
    \end{align*}
  \item[Horizontal shift:] For $y\in X$, define $L^y = \{x\mapsto u(x+y): u\in L\}$ and $X^y = \{x-y: x\in X\}$. At $x\in X^y$, the $L^y$ subdifferential of the function $f^y(x) \coloneqq f(x+y)$ is $\partial_{L^y}f^y(x) = \{x\mapsto u(x+y): u\in \partial_L f(x+y)\}$. Furthermore $f^y$ is $(L^y,X^y)$-convex if and only if $f$ is $(L,X)$-convex.

  Indeed,
    \begin{align*}
    \partial_{L^y} f^y(x)&=\{u\in L^y:f^y(z)\geq f^y(x)+u(y)-u(x), \forall z\in X^y \}\\
    &=\{\begin{multlined}[t]u(\cdot +y): u\in L,\\f(z+y)\geq f(x+y)+u(z+y)-u(x+y), \forall z\in X^y\}\end{multlined}\\ 
    &=\{u(\cdot+y), u\in L:f(t)\geq f(x+y)+u(t)-u(x+y), \forall t\in X\}\\
    &= \{u(\cdot+y):u\in\partial_Lf(x+y)\}
    \end{align*}
  \end{description}
\end{remark}

\begin{proposition}\label{prop:maximumrule}
  Let $G\subset L$ be a finite subset of $L$ and let $f=\max_{g\in G} g$. Define $A(x) = \{g\in G: g(x) = f(x)\}$. Then we have:

  \begin{equation}\label{eq:maxrule}\partial_{L_x} f(x) \supset \co_{L_x} A(x).\end{equation}
\end{proposition}

\begin{proof}
  Let $h = \max_{g\in A(x)} (g-g(x))$, and let $U=\supp{L_x}{h} = \co_{L_x} \{g-g(x): g\in A(x)\}$. By definition of $L_x$ (see Formula~\eqref{eq:Lx}), we know that $h(x)=0$. Let $l\in L$ be such that $l_x = l-l(x) \in U$. Then, we have for any $y\in X$:

  \begin{align*}
  l(y) - l(x) &\leq h(y)\\
                  &\leq \sup_{g\in A(x)} g(y) - g(x)\\
                  &= \sup_{g\in A(x)} g(y) - f(x)\\
  f(x) + l(y) - l(x) &= \sup_{g\in A(x)} g(y) = f(y)
  \end{align*}

  And therefore $l_x\in \partial_{L_x} f(x)$.
\end{proof}

The reverse inclusion in Proposition~\ref{prop:maximumrule} is generally not true, as shown by the following example.
\begin{example}
Let $L$ be the set of all monotonous functions over $\R$ vanishing at 0. Consider the $L$-convex function $\max(g_1(x),g_2(x))$, where $g_1(x)=x$ and $g_2(x)=-x$ at $x=1$. At this point, $A(1) = \{g_1\}$. The function $u(x) = \max(0,x)-1$ is an $L_x$-subgradient of $f$ at $x=1$ (See Figure~\ref{fig:counterExampleMax}), since for any $y\in \R$, we have $f(y) = \abs{y} \geq 1 + (\max(0,y) -1) - 0 =f(x) + u(y)-u(x)$.  However, $u$ does not lie in $\co_{L_x} (t\mapsto g_1(t)-t)$, since for $y=-1$, we have $g_1(y) - 1 = -1-1 < 1 -1 -1 +0 = g_1(x) + u(-1) - u(1)$. Therefore $\max(0,x)\in \partial_{L_1} f(1)\setminus \partial_{L_1} g_1(1)$.

\begin{figure}
  \begin{center}
    \begin{tikzpicture}
      \begin{axis}[axis lines=center,ytick={},xtick={1}]
      \addplot[thin,dashed,domain=-2:0] {x};
      \addplot[thin,dashed,domain=0:2] {-x};
      \addplot[domain=-2:2] {abs(x)};\addlegendentry{$\max(g_1,g_2)$};
      \addplot[very thick,domain=-2:0] {-1};
      \addplot[very thick,domain=0:2] {x-1};\addlegendentry{$u$};
      \end{axis}
    \end{tikzpicture}
  \end{center}
  \caption{The function $u$ is a $L_1$-subgradient of the function $|x|$ at $x=1$.}\label{fig:counterExampleMax}
\end{figure}
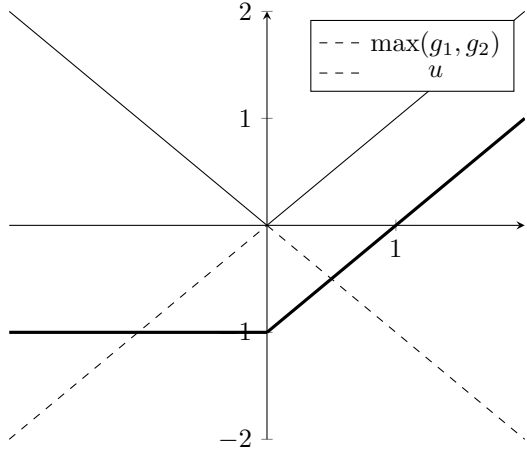
\end{example}

Consider two arbitrary sets $X$ and $Y$. For a set $G\subset \F{X}$ and a function $u: X\to Y$, we define $G\circ u = \{g\circ u: g\in G\}$. It is clear that $G\circ u\subset \F{Y}$.

\begin{proposition}[Composition Rule]\label{prop:composition}
  Let $f$ be an $(L,X)$-convex function, and $u:Y\to X$. Then
  \[ (\partial_L f(x)) \circ u \subset \partial_{L\circ u} (f\circ u) (x), \]
  where $L\circ u = \{l\circ u: l\in L\}$.

  Furthermore, if $u$ is onto the domain of $f$, then we have equality:
  \[ (\partial_L f(x)) \circ u = \partial_{L\circ u} (f\circ u) (x), \]
\end{proposition}

\begin{proof}
  Let $w\in \partial_L f(u(x)) \circ u$, that is, there exists $v\in \partial_L f(u(x))$ such that $w=v\circ u$. Then, we have, for any $y\in X$:
  \[f(y)\geq f(u(x)) + v(u(y)) - v(u(x)).\]
  This is true, in particular, when $y=u(z)$ for some $z\in Y$. In other words, for any $z\in Y$,
  \[f(u(z))\geq f(u(x)) + v(u(y)) - v(u(x))\]
  and therefore $w\in \partial_{L\circ u} f\circ u(x)$.

  When the image of $u$ $\im u = \dom f$, and $w\in \partial_{L\circ u} f\circ u(x)$, letting $w=v\circ u$, then for any $y\in \dom f$, there exists $z\in Y$ such that $y=u(z)$. Then:
  \[f(y) = f(u(z))\geq f(u(x)) + v(u(y)) - v(u(x)).\]
  So, $v\in \partial_L f(u(x))$, which proves the reverse inclusion.
\end{proof}

We now turn our attention to the sum of abstract convex functions.
  We start by recalling the definition of infimal convolution.
  \begin{definition}\label{def:infimal}[Infimal Convolution\cite{jeyakumar.ea:2007}]
    Let $f\in \F{X_1}$ and $g\in \F{X_2}$ be real-valued functions, where $X_1, X_2$ are sets such that the Minkowski sum $X_1+X_2:=\{x+y, x\in X_1,y\in X_2\}$ is well defined. The infimal convolution $f\oplus g:(X_1+X_2)\to \R$ is defined as:
    \[
    (f\oplus g)(x) = \inf_{\substack{x_1+x_2=x\\x_1\in X_1\\x_2\in X_2}}\big( f(x_1) +g(x_2)\big).
    \]
    The convolution is said to be \emph{exact} when the infimum is attained for every $x\in X$.
  \end{definition}
  Note that unlike in most definitions, we do not assume $f$ and $g$ to be defined on the same domain. This is so we can apply the infimal convolution to the conjugates of $L_1$- and $L_2$-convex functions for arbitrary families of functions $L_1$ and $L_2$.

  We now give some results on the sums of abstract convex sets and abstract convex functions. In the rest of this section we let $X$ be an arbitrary set, $f_1\in \F{X}$ and $f_2\in\F{X}$ be two functions defined on $X$, and $L_1$ and $L_2$ any two families of functions on $X$.

  \begin{proposition}\label{prop:sumconvexfunctions}
    If $f_1$ is $L_1$-convex and $f_2$ is $L_2$-convex, then $f_1+f_2$ is $L_1+L_2$-convex.
  \end{proposition}

  \begin{proof}
    Let $U_i = \supp{L_i}{f_i}$ for $i\in \{1,2\}$.
     We have that \begin{align*}
     (f_1+f_2)(x) &= \sup \{u_1(x): x\in U_1\} + \sup \{u_2(x): x\in U_2 \}\\
          &= \sup \{ (u_1(x) + u_2(x) : u_1\in U_1, u_2\in U_2\}.\qedhere
     \end{align*}
  \end{proof}

  \begin{proposition}\label{prop:convexhullsumsets} Let $f_1$, $f_2$, $L_1$ and $L_2$ be as defined above. Then, 
  \[\supp{L_1+L_2}{f_1+f_2} = \co_{L_1+L_2}\Big( \supp{L_1}{f_1}+\supp{L_2}{f_2}\Big)\].
  \end{proposition}
  \begin{proof}
    We have
    \begin{align*}
      \co_{L_1+L_2}(f_1+f_2)(x) &= \co_{L_1+L_2}\left(\sup_{l_1\in \supp{L_1}{f_1}} l_1(x) +\sup_{l_2\in \supp{L_2}{f_2}} l_2(x)\right)\\
      \sup_{l\in \supp{L_1+L_2}{f_1+f_2}} l(x)&= \co_{L_1+L_2}\sup_{l_1+l_2\in \supp{L_1}{f_1}+\supp{L_2}{f_2}} l_1(x) +l_2(x)\\
      &= \sup_{l_1+l_2\in \supp{L_1}{f_1}+\supp{L_2}{f_2}} l_1(x) +l_2(x)
    \end{align*}
    Applying Proposition~\ref{prop:convexhullsup} gives the result.
  \end{proof}

  Define the \emph{strict epigraph} of $f$ by $$\epi^{\mathrm{st}}_X f := \{(x,c)\in X\times \R:f(x)<c\}$$
  \begin{proposition}\label{prop:epiconvolution}
  The following statements are true.
      \begin{enumerate}
      \item $\epi^{\mathrm{st}}_{L_1+L_2}({f_1}^*_{L_1}\oplus {f_2}^*_{L_2}) = \epi^{\mathrm{st}}_{L_1} {f_1}^*_{L_1} + \epi^{\mathrm{st}}_{L_2} {f_2}^*_{L_2}$
      \item
      \begin{equation}\label{eq:sumepigraphs}
      \epi_{L_1+L_2}({f_1}^*_{L_1}\oplus {f_2}^*_{L_2}) \supseteq  \epi_{L_1} {f_1}^*_{L_1} + \epi_{L_2} {f_2}^*_{L_2}
      \end{equation}
      and equality holds if and only if the infimal convolution ${f_1}^*_{L_1}\oplus {f_2}^*_{L_2}$ is exact at every $l\in \dom ({f_1}^*_{L_1}\oplus {f_2}^*_{L_2})$.
      \item $\co_{H_{L_1+L_2}}\epi_{L_1+L_2}({f_1}^*_{L_1}\oplus {f_2}^*_{L_2}) =  \co_{H_{L_1+L_2}}\epi_{L_1} {f_1}^*_{L_1} + \epi_{L_2} {f_2}^*_{L_2}$
    \end{enumerate}

  \end{proposition}
  \begin{proof}
  \begin{enumerate}
    Suppose that $(l_1,c_1)\in \epi_{L_1}{f_1}^*_{L_1}$ and $(l_2,c_2)\in \epi_{L_2}{f_2}^*_{L_2}$. Then $({f_1}^*_{L_1}\oplus {f_2}^*_{L_2})(l_1+l_2)\leq {f_1}^*_{L_1}(l_1)+{f_2}^*_{L_2}(l_2)\leq c_1+c_2$. Therefore $(l_1+l_2,c_1+c_2)\in \epi_{L_1+L_2} ({f_1}^*_{L_1}\oplus {f_2}^*_{L_2})$ and so
      \[ \epi_{L_1}{f_1}^*_{L_1} + \epi_{L_2}{f_2}^*_{L_2} \subset \epi_{L_1+L_2} ({f_1}^*_{L_1}\oplus {f_2}^*_{L_2}). \]

      Now suppose that these two sets are equal, and let $l\in \dom ({f_1}^*_{L_1}\oplus {f_2}^*_{L_2})$. Then $(l,({f_1}^*_{L_1}\oplus {f_2}^*_{L_2})(l)) \in \epi_{L_1+L_2}({f_1}^*_{L_1}\oplus {f_2}^*_{L_2}) =  \epi_{L_1} {f_1}^*_{L_1} + \epi_{L_2} {f_2}^*_{L_2}$. Let $(l_1,c_1)\in \epi_{L_1} {f_1}^*_{L_1}$ and $(l_2,c_2)\in \epi_{L_2} {f_2}^*_{L_2}$ be such that $(l,({f_1}^*_{L_1}\oplus {f_2}^*_{L_2})(l)) = (l_1,c_1)+(l_2,c_2)$. We have:

      \begin{align*}
      \inf_{u_1+u_2=l} {f_1}^*_{L_1}(u_1)+{f_2}^*_{L_2}(u_2) &= ({f_1}^*_{L_1}\oplus {f_2}^*_{L_2})(l) \\&= c_1+c_2 \geq {f_1}^*_{L_1}(l_1)+{f_2}^*_{L_2}(l_2)\\ &\geq \inf_{u_1+u_2=l} {f_1}^*_{L_1}(u_1)+{f_2}^*_{L_2}(u_2) 
      \end{align*}
      and so the infimal convolution is attained at $(l_1,l_2)$.

      Conversely, assume that the infimal convolution is attained, and consider $(l,c)\in \epi_{L_1+L_2}({f_1}^*_{L_1}\oplus {f_2}^*_{L_2})$. There exist $l_1\in L_1$and $l_2\in L_2$ be such that $l_1+l_2=l$ and ${f_1}^*_{L_1}(l_1) + {f_2}^*_{L_2}(l_2) = ({f_1}^*_{L_1}\oplus {f_2}^*_{L_2})(l) \leq c$.

      Define $c_1 = {f_1}^*_{L_1}(l_1)$ and $c_2 = c-c_1$. It is clear that $(l_1,c_1)\in (l_1,c_1)\in \epi_{L_1} {f_1}^*_{L_1}$. Since $f_2(l_1) \leq c- f_1(l_1) = c_2$, we have $(l_2,c_2)\in \epi_{L_2} {f_2}^*_{L_2}$. Therefore, $(l,c) = (l_1,c_1)+(l_2,c_2)\in \epi_{L_1} {f_1}^*_{L_1} + \epi_{L_2} {f_2}^*_{L_2}$, which proves the result.

      \item Let
      \begin{align*}
        \lambda(x) &= \sup_{(u,\alpha)\in \epi_{L_1} {f_1}^*_{L_1} +\epi_{L_2} {f_2}^*_{L_2}} u(x)- \alpha\\
        \gamma(x) &= \sup_{(u,\alpha)\in \epi_{L_1+L_2}} ({f_1}^*_{L_1} \oplus {f_2}^*_{L_2}) u(x)- \alpha
      \end{align*}
      By~\eqref{eq:sumepigraphs}, it is clear that $\lambda\leq \gamma$. Take $(u,\alpha)\in \epi_{L_1+L_2} ({f_1}^*_{L_1} \oplus {f_2}^*_{L_2})$. We have:
      \begin{align*}
        u(x) - \alpha &\leq u(x) - \inf_{u_1+u_2=u} {f_1}^*_{L_1}(u_1) + {f_2}^*_{L_2}(u_2)\\
        &= \sup_{u_1+u_2=u}u(x)-{f_1}^*_{L_1}(u_1) - {f_2}^*_{L_2}(u_2)\\
        &\leq \sup_{\substack{u_1+u_2=u\\a\geq {f_1}^*_{L_1}(u_1)\\b\geq {f_2}^*_{L_2}(u_2)}}u(x)-a-b\\
        &\leq \sup_{\substack{(u_1,a)\in \epi_{L_1}{f_1}^*_{L_1}\\(u_2,b)\in \epi_{L_2}{f_2}^*_{L_2}}}u(x)-a-b\\
        &= \lambda(x)
      \end{align*}
      and therefore $\gamma\leq \lambda$, which proves the result.
    \end{enumerate}
  \end{proof}

  \begin{corollary}\label{cor:suminfimalconvolution}
  If the infimal convolution is exact at every $x\in \dom (f_1\oplus f_2)$ then
    $(f_1+f_2)^*_{L_1+L_2} = \co_{L_1+L_2} ({f_1}^*_{L_1}\oplus {f_2}^*_{L_2})$
  \end{corollary}
  \begin{proof}
    By Propositions~\ref{prop:epiconjugate} and~\ref{prop:epiconvolution}, we have that
      \begin{equation}
        \epi_{L_1+L_2}({f_1}^*_{L_1}\oplus {f_2}^*_{L_2}) =  \supp{H_{L_1}}{f_1} + \supp{H_{L_2}}{f_2}
 \label{eq:convexepigraph}
      \end{equation}

    Applying Proposition~\ref{prop:convexhullsumsets}, this implies that
    \begin{align*}
      \co_{H_{L_1}+H_{L_2}}\epi_{L_1+L_2}({f_1}^*_{L_1}\oplus {f_2}^*_{L_2}) &=  \supp{H_{L_1}+H_{L_2}}{f_1+f_2}\nonumber\\
      &= \epi (f_1+f_2)^*_{L_1+L_2}
    \end{align*}
    The result then follows from Corollary~\ref{cor:epigraphhull}.
  \end{proof}

  \begin{corollary}\label{cor:equivalenceconvexity}
    The following statements are equivalent:
    \begin{enumerate}[label=(\roman*)]
      \item\label{item:sumconvex} $\supp{H_{L_1}+H_{L_2}}{f_1+f_2} = \supp{H_{L_1}}{f_1}+\supp{H_{L_2}}{f_2}$
      \item\label{item:sumconjugates} $({f_1}^*_{L_1}\oplus {f_2}^*_{L_2}) = (f_1 + f_2)^*_{L_1+L+2}$ with exact infimal convolution.
      \item\label{item:sumepigraphs} $\epi_{L_1} {f_1}^*_{L_1} +\epi_{L_2}{f_2}^*_{L_2} = \epi_{L_1+L_2} (f_1+f_2)^*_{L_1+L_2}$.
    \end{enumerate}
  \end{corollary}
  \begin{proof}
    $\ref{item:sumconjugates}\implies \ref{item:sumconvex}$ follows
    directly from Equation~\eqref{eq:convexepigraph}, since the $(H_{L_1}+H_{L_2})$-convexity of the left-hand side is equivalent to the convexity of the right-hand side.

    To prove $\ref{item:sumconvex}\implies \ref{item:sumconjugates}$, take any $l_1\in L_1$, $l_2\in L_2$ and $l=l_1+l_2$. We have:
    \begin{align*}
    {f_1}^*_{L_1}(l_1)+{f_2}^*_{L_2}(l_2) &= \sup_{x\in X} \big(l_1(x) - f_1(x)\big) + \sup_{x\in X} \big(l_2(x) - f_2(x)\big)\\
    &\geq \sup_{x\in X} \big(l(x) - f_1(x) - f_2(x)\big) = (f_1+f_2)^*_{L_1+L_2}(l)
    \end{align*}
    Therefore $(f_1+f_2)^*_{L_1+L_2}(l)\leq \inf_{l_1+l_2=l}({f_1}^*_{L_1}(l_1)+{f_2}^*_{L_2}(l_2) = ({f_1}^*_{L_1}\oplus {f_2}^*_{L_2})(l)$.

    On the other hand,
    \begin{align*}
      (l,(f_1+f_2)^*_{L_1+L_2}(l))&\in \epi_{L_1+L_2} (f_1+f_2)^*_{L_1+L_2} \\
      &= \supp{H_{L_1+L_2}}{f_1+f_2} = \supp{H_{L_1}}{f_1}+\supp{H_{L_2}}{f_2}\\
     &= \epi({f_1}^*_{L_1},L_1)+\epi({f_2}^*_{L_2},L_2).
    \end{align*}
    Therefore there exist $(l_1,c_1)\in \epi_{L_1}{f_1}^*_{L_1}$ and $(l_2,c_2)\in \epi_{L_2}{f_2}^*_{L_2}$ such that $l_1+l_2=l$, $c_1+c_2=(f_1+f_2)^*_{L_1+L_2}(l)$. Then
    \begin{align*}
    (f_1+f_2)^*_{L_1+L_2}(l) = c_1+c_2 &\geq {f_1}^*_{L_1}(l_1)+ {f_2}^*_{L_2}(l_2) \\&\geq \inf_{u_1+u_2=l}{f_1}^*_{L_1}(u_1)+ {f_2}^*_{L_2}(u_2) = ({f_1}^*_{L_1} \oplus {f_2}^*_{L_2})(l)
    \end{align*}
    Thus $(f_1+f_2)^*_{L_1+L_2} = ({f_1}^*_{L_1}\oplus {f_2}^*_{L_2})$, and the infimal convolution is exact.

    That $\ref{item:sumconjugates} \iff \ref{item:sumepigraphs}$ follows from Proposition~\ref{prop:epiconvolution}.
  \end{proof}

  Note that the statements above can be interpreted as follows: the sum of abstract convex sets (as the support sets of two abstract convex functions) is abstract convex if and only if the infimal convolution of the two functions is abstract convex.
  %\begin{enumerate}
  %  \item The sum of a $L_1$-convex and a $L_2$-convex sets is $L_1+L_2$-convex.
  %  \item The conjugate of the infimal convolution of a $L_1$-convex and a $L_2$-convex functions is $(X,L_1+L_2)$-convex.
  %\end{enumerate}
  
Define the support function of the set $C\subset X$ as $\sigma_{C}:L\rightarrow \R$, $\sigma_C(l):=\sup_{x\in C}l(x)$ and the indicator function of $C$ by $\delta_C(x)\coloneqq 0$ if $x\in C$ and $\delta_C(x)\coloneqq +\infty$ if $x\notin C$. It is easy to see that $\sigma_C = {\delta_C}^*_L$.

  \begin{corollary}
    Let $C\subset X$ and $D\subset X$ be subsets of $X$. Then, \[
    \epi_{L_1+L_2} (\sigma_{C\cap D}) = \co_{(L_1+L_2)\times \R}\big(\epi_{L_1} \sigma_C + \epi_{L_2} \sigma_D\big).
  \] 
  \end{corollary}
  \begin{proof}
    We have from Corollary~\ref{cor:suminfimalconvolution}
    \[
      \sigma_{C\cap D} = (\delta_C+\delta_D)^*_{L_1+L_2} = \co_{L_1+L_2} ({\delta_C}^*_{L_1} \oplus {\delta_D}^*_{L_2})
    \]
    Applying Corollary~\ref{cor:epigraphhull}, this implies that
    \begin{align*}
    \epi_{L_1+L_2} \sigma_{C\cap D} &= \co_{(L_1+L_2)\times \R} \epi_{L_1+L_2}({\delta_C}^*_{L_1}\oplus{\delta_D}^*_{L_2})\\
    &= \co_{(L_1+L_2)\times \R}(\epi_{L_1}{\delta_C}^*_{L_1} +\epi_{L_2}{\delta_D}^*_{L_2})
    \end{align*}
    where the latter equality follows from Proposition~\ref{prop:epiconvolution}. From there and the fact that ${\delta_C}^*_{L_1} = \sigma_C$ and ${\delta_D}^*_{L_2}=\sigma_D$ we conclude that
  \[
    \epi_{L_1+L_2} (\sigma_{C\cap D}) = \co_{(L_1+L_2)\times \R}(\epi_{L_1} \sigma_C + \epi_{L_2} \sigma_D).\qedhere
  \]
  \end{proof}

Note that the definition of the support function implies $N_{L}(x,C)=\{l\in L: \sigma_C(l)=l(x)\}$.

\begin{proposition}\label{prop:sumnormal}
Consider $C,D\subset X$ such that $C$ is $(X,L_1)$-convex and $D$ is $(X,L_2)$-convex and suppose that $\supp{H_{L_1}}{\delta_C} + \supp{H_{L_2}}{\delta_D}=\supp{H_{L_1+L_2}}{\delta_{C} + \delta_{D}}$
Then for all $x\in X$, $N_L(x,C\cap D)=N_{L_1}(x,C)+N_{L_2}(x,D)$.
\end{proposition}
\begin{proof}
First, if $x\notin C\cap D$, that is, $x\notin C$ or $x\notin D$, then the abstract normal cones on both sides of the equality are both the empty set.

 Now consider $x\in C\cap D$. Take $l\in N_{L_1}(x,C)+N_{L_2}(x,D)$, then there exists $l_C\in N_{L_1}(x,C)$ and $l_D\in N_{L_2}(x,D)$ such that $l=l_C+l_D$. For all $y\in C$ $l_C(y)\leq l_C(x)$ and for all $z\in D$ $l_D(z)\leq l_D(x)$, then taking any point $w\in C\cap D$ we have, summing the inequalities above, $l_C(w)+l_D(w)\leq l_C(x)+l_D(x)$. So, for all $w\in C\cap D$, $l(w)\leq l(x)$, implying that $l\in N_L(x,C\cap D)$.

  On the other hand, consider $l\in N_L(x,C\cap D)$.
  We have, by assumption that $\supp{H_{L_1}}{\delta_C} + \supp{H_{L_2}}{\delta_D}=\supp{H_{L_1+L_2}}{\delta_{C} + \delta_{D}}$ and Corollary~\ref{cor:equivalenceconvexity}, that
  \[
  (\sigma_C \oplus \sigma_D)(l) = (\delta_C+\delta_D)^*_{L_1+L_2}(l) = {\delta_{C\cap D}}^*_{L_1+L_2}(l) = \sigma_{C\cap D}(l) = l(x).
  \]
  and since the infimal convolution is exact, there exist $l_1\in L_1$ and $l_2$ in $L_2$ with $l_1+l_2=l$ such that $\sigma_C(l_1)+\sigma_D(l_2) = l(x)$.

Now, for all $z\in D$,
  \[0\geq l_1(x)-\sigma_C(l_1)=(l-l_2)(x)+\sigma_D(l_2)-l(x)=\sigma_D(l_2)-l_2(x)\geq l_2(z)-l_2(x).\]
 This implies that $l_2\in N_{L_2}(x,D)$. In the same way we can prove that $l_1\in N_{L_1}(x,C)$. Then $N_L(x,C\cap D)\subseteq N_{L_1}(x,C)+N_{L_2}(x,D)$. This proves the equality. 
\end{proof}

%\begin{proposition}[Proposition 3.1 in \parencite{reg-jeya}]
%Let $C$ and $D$ be closed and convex sets of $X$. Then the set $(\epi \, \sigma_C+\epi \, \sigma_D)$ is weak$^*$ closed if one of the following conditions holds:
%\begin{itemize}
%    \item[(i)] $(\inter D)\cap C\neq \emptyset$.
%    \item[(ii)] $0\in core(C-D)$.
%    \item[(iii)] $cone(C-D)$ is a closed subspace.
%\end{itemize}
%\end{proposition}

%\todonote{Now we cannot use the intersection of the domains is nonempty. We need something different.}

  \begin{theorem}[Sum Rule]\label{thm:sumsubdif}
  Consider two families of functions $L_1\subset \F{X}$ and $L_2\subset \F{X}$, and two functions $f_1\in\F{X}$ and $f_2\in\F{X}$ be $L_1$-convex and $L_2$-convex respectively such that $\dom f_1 \cap \dom f_2 \neq \emptyset$.  Assume that $\supp{H_{L_1}}{f_1}+\supp{H_{L_2}}{f_2} = \supp{H_{L_1}+H_{L_2}}{f_1+f_2}$. Then for all $x\in \dom f_1 \cap \dom f_2$, we have the following:
   \[ \partial_{L_1+L_2} (f_1+f_2) = \partial_{L_1} f_1 (x) + \partial_{L_2} f_2(x).\]
  \end{theorem}

  \begin{proof}
    Take $u_1\in \partial_{L_1} f_1 (x)$ and $u_2\in \partial_{L_2} f_2 (x)$. Then, for any $y\in X$,
    \begin{align*}
      f_1(y) + f_2(y)
      &\geq f_1(x) + u_1(y) -u_1(x) + f_2(x)+ u_2(y) -u_2(x)\\
      &= f_1(x)+f_2(x) + (u_1+u_2)(y) - (u_1+u_2)(y) 
    \end{align*}
    Therefore $u_1+u_2\in \partial_{L_1+L_2} (f_1+f_2)(x)$.
    
  For the reverse inclusion, take $u\in \partial_{L_1+L_2}f(x)$ and consider the sets.
    \begin{align*}
      \Omega_1&\coloneqq \{(x,\lambda_1,\lambda_2)\in X\times\R\times \R: \lambda_1\geq f_1(x)\}\\
      \Omega_2&\coloneqq \{(x,\lambda_1,\lambda_2)\in X\times\R\times \R: \lambda_2\geq f_2(x)\}.
  \end{align*}

  For $u_1\in L_1$, $\sigma_{\Omega_1}(u_1,-1,0)=\sup_{x\in X,\lambda\geq f(x)} u_1(x) - \lambda=\sup_{x\in X} u_1(x) - f_1(x) = {f_1}^*_{L_1}(u_1)$. Similarly, $\sigma_{\Omega_2}(v,0,-1) = {f_2}^*_{L_2}(v)$. Therefore, $\sigma_{\Omega_1}(u,-1,0) +\sigma_{\Omega_2}(v,0,-1) = {f_1}^*_{L_1}(u) + {f_2}^*_{L_2}(v)$.

  By assumption we have: $\supp{H_{L_1}}{f_1} + \supp{H_{L_2}}{f_2} = \supp{H_{L_1}+H_{L_2}}{f_1+f_2}$, and, by Corollary~\ref{cor:equivalenceconvexity}, there exist $l_1\in L_1$ and $l_2\in L_2$ such that $(f_1+f_2)^*_{L_1+L_2}(l_1+l_2) = {f_1}^*_{L_1}(l_1)+{f_2}^*_{L_2}(l_2) = \sigma_1(l_1,-1,0)+\sigma_2(l_2,0,-1)$.

  In particular, $\sigma_1\oplus \sigma_2$ is $\big((L_1+L_2) \times (\{-1\}+\{0\}) \times (\{0\}+\{-1\})\big)$-convex. Applying Corollary~\ref{cor:suminfimalconvolution} this means that $(\delta_{\Omega_1}+\delta_{\Omega_2})^*_{\big((L_1+L_2) \times (\{-1\}+\{0\}) \times (\{0\}+\{-1\})\big)} = {\delta^*_{\Omega_1}}_{L_1\times \{-1\} \times \{0\}}\oplus {\delta^*_{\Omega_2}}_{L_2\times \{0\} \times \{-1\}}$. From this we can apply Corollary\ref{cor:equivalenceconvexity} to conclude that $\supp{H_{L_1}}{\delta_{\Omega_1}}+\supp{H_{L_2}}{\delta_{\Omega_2}} = \supp{H_{L_1}+H_{L_2}}{\delta_{\Omega_1}+\delta_{\Omega_2}}$.

 Fixing $x\in \dom(f_1)\cap \dom(f_2)$ and considering $\lambda=\max(f_1(x),f_2(x))$, then $(x,\lambda,\lambda)\in \Omega_1\cap\Omega_2\neq\emptyset$. We claim that $(u,-1,-1)\in N_L\left((x,f_1(x),f_2(x)),\Omega_1\cap \Omega_2\right)$. Indeed for all $(y,\lambda,\alpha)\in \Omega_1\cap\Omega_2$, since $u\in\partial_L f(x)$, we have $$u(y)-u(x)+(-1)(\lambda-f_1(x))+(-1)(\alpha-f_2(x))\leq f_1(y)-\lambda+f_2(y)-\alpha\leq 0.$$
 
  Invoking Proposition \ref{prop:sumnormal}, we obtain that 
  \begin{multline*}N_L\left((x,f_1(x),f_2(x)),\Omega_1\cap \Omega_2\right)=\\N_{L_1}\left((x,f_1(x),f_2(x)),\Omega_1\right)+N_{L_2}\left((x,f_1(x),f_2(x)),\Omega_2\right).\end{multline*}
  Now, note that for all  $(x,\lambda,\alpha)\in N_L\left((x,f_1(x),f_2(x)),\Omega_1\right)$, we have $\alpha=0$, and $(x,\lambda,\alpha)\in N_L\left((x,f_1(x),f_2(x)),\Omega_2\right)$ implies that $\lambda=0$. Then $(v,-1,-1)=(u_1,\lambda,0)+(u_2,0,\alpha)$, and 
  $u=u_1+u_2$ and $\lambda=\alpha=-1$. Using Proposition \ref{normalsubdif}, we have that $u_1\in\partial_L f_1(x)$ and $u_2\in \partial_L f_2(x)$ as desired. 
  \end{proof}

  Proposition~\ref{thm:sumsubdif} a generalises the conditions obtained by~\textcite{burachik.ea:2005} in the context of the classical lower semi-continuous convex functions (recall that lower semi-continuous convex functions are abstract convex with respect to linear functions, and that abstract convex sets are closed convex sets~\parencite{rubinov:2000} so that the abstract convex hull of a set is its weak* closure.)  They also generalise the results by~\cite{jeyakumar.ea:2007} in the case when the set $L$ is additive and $0\in L$.

  \section{The abstract subdifferential as an abstract monotone operator}\label{sec:monotoneoperator}
  This section is dedicated to the study of the abstract subdifferential of an L-convex function. Following we return to the definition of abstract monotonicity. From here onward we assume that the set $X$ is a Banach space.
  \begin{definition}[Abstract monotonicity~{\parencite[Eq. (1.1.8)]{FoundationsOfPallas1997}}] Let $L\subset \F{X}$ be a set of abstract linear functions. An operator $T: X \mto L$ is said to be $L-$abstract monotone with respect to $L$ if for any $x,y\in X$, $u\in T(x)$, $v\in T(y)$, $u(x) - u(y) + v(y) -v(x)\geq 0$.

  Furthermore, $T$ is \emph{maximal} if it is maximal with respect to set inclusion:  if $(y,v)\in X\times L$, if $u(x) - u(y) + v(y) -v(x)\geq 0$ for all $(x,u)\in T$, then $v\in T(y)$.
  \end{definition}
The inverse operator, denoted by $T^{-1}:L\rightrightarrows X$, is defined by $T^{-1}(u):=\{x\in X: u\in T(x)\}$.
  \begin{proposition}
  Consider two abstract monotone operators $T_1:X\to L_1$ and $T_2:X\to L_2$ . Then
  \begin{enumerate}
    \item $T_1^{-1}$ is abstract monotone; furthermore if $T_1$ is $X-$abstract maximal monotone, then so is $T_1^{-1}$.
    \item For any $\lambda_1\geq 0, \lambda_2\geq 0$, $\lambda_1 T_1 + \lambda_2 T_2$ is $\lambda_1 L_1 + \lambda_2L_2-$abstract monotone.
  \end{enumerate}
  \end{proposition}
  \begin{proof}
   \begin{enumerate}
    \item Let $u,v \in L_1$ and $x,y\in X$ be such that $x\in T_1^{-1}(u)$ (that is, $u\in T_1(x)$) and $y\in T_1^{-1}(v)$ (that is, $v\in T_1(y)$). Then it is clear that $u(x) - u(y) + v(y) -v(x)\geq 0$, implying that $T_1^{-1}$ is abstract monotone. Now consider a pair $(u,x)\in L_1\times X$ such that for all $(v,y)\in L_1\times X$ such that $y\in T^{-1}(v)$, $u(x)-u(y)+v(y)-v(x)\geq 0$. If $T_1$ is maximal, then we have that $u\in T_1(x)$, that is, $x\in T_1^{-1}(u)$, proving the maximality of $T^{-1}$.
      \item Let $x,y\in X$ and $u\in (\lambda_1 T_1 + \lambda_2 T_2)(x)$ and $v\in (\lambda_1 T_1 + \lambda_2 T_2)(v)$. Then there exist $u_1\in T_1(x)$ and $u_2\in T_2(x)$ such that $u=\lambda_1u_1 + \lambda_2u_2$ and  $v_1\in T_1(y)$ and $v_2\in T_2(y)$ such that $v=\lambda_1v_1 + \lambda_2v_2$. We have:
      \begin{multline*}
        u(x) - u(y) +v(y) - v(x) = \lambda_1 (u_1(x) - u_1(y) + v_1(y) - v_1(x)) \\+ \lambda_2 (u_2(x) - u_2(y) + v_2(y) - v_2(x))\geq 0.
      \end{multline*}
    \end{enumerate}
  \end{proof}

Denote the set of linear functions by $\mathcal{L}$, then we state the following Assumption.
  \begin{assumption}\label{ass:bounded}
      Given a L-convex function $f\in\F{X}$, for each $a>0$ and any $x\in X$, $\supp{H_L}{f} + \supp{H_{\mathcal{L}}}{a\|\cdot-x\|}$
is $H_{L+\mathcal{L}}$-convex,
    and the function $f+a\|\cdot-x\|$ is bounded from below on the domain of $f$.
  \end{assumption}

  The next result is a generalisation of the Br{\o}nsted-Rockafellar theorem, which is important to prove the main result of this section.

  \begin{theorem}[Br{\o}nsted-Rockafellar theorem extended]\label{thm:bronstedrockafellar}
  Let $f$ be a proper lsc $L$-convex function, $y\in \dom f$ satisfying Assumption~\ref{ass:bounded} and $v\in \dom f^*_L$, and let $\lambda\geq 0$ and $\mu\geq 0$ be such that $f(y) + f^*_L(v) \leq v(y) + \lambda\mu$. Then, there exists $z\in X$ and $w\in \partial_L f(z)$ such that $\|z-y\|\leq \lambda$ and $w-v = \langle p,\cdot\rangle$, where $p\in X^*$ such that $\|p\|\leq \mu$.
  \end{theorem}

  \begin{proof}
    Set $\alpha = \lambda\mu$. We can assume that $\alpha>0$ as otherwise, by Proposition \ref{prop:moreau}, we have $v\in \partial_L f(y)$ and we are done. Define $h = f-v$. It can be seen, from Remark~\ref{rmk:diffconvexlinear}, that $h$ is $L-\{v\}$-convex, and for any $x\in X$, $\partial_{L-\{v\}} h(x) = \partial_L f(x) - \{v\}$. For any $x\in X$, according to \eqref{eq:Moreau} we have that $$h(x) = -(v(x)-f(x))\geq -f^*_L(v)\geq f(y)-v(y) -\lambda\mu = h(y)-\alpha.$$
    Hence, $h$ is bounded from below and $\alpha + \inf_{x\in X} h(x)\geq h(y)$.

    Applying Ekeland's principle (Theorem 1.1 in \cite{ekeland}) to $h$, we get the existence of $z\in X$ such that $\|z-y\|\leq \lambda$ and $z=\argmin(h + \frac{\alpha}{\lambda}\|\cdot - z\|)$. Define $\tilde{h}:=h + \frac{\alpha}{\lambda}\|\cdot - z\|$.

    By Proposition~\ref{prop:sumconvexfunctions}, $\tilde{h}$ is in $((L-v)+\mathcal{L})$-convex, where $\mathcal{L}$ is the set of linear functions, and so $0\in \partial_{L-\{v\}+\mathcal{L}} \tilde{h}(z)$.
    Since Assumption~\ref{ass:bounded} holds, we can apply Theorem~\ref{thm:sumsubdif} to find that
    $$ 0\in \partial_L f(z) - \{v\} + \mu \{\langle p,\cdot\rangle: p\in B(0,1)\}. $$

    Hence, there exists $w\in \partial_L f(z)$ such that $v-w\in \{\langle p,\cdot\rangle: \|p\|\leq \mu\}$.
  \end{proof}
  
  \begin{remark}
  Note that $\{v\}\subseteq \{v + \langle p,\cdot\rangle: \|p\|\leq \mu\}\cap L \neq \emptyset$. When equality occurs, this implies that $v=w$ in the theorem above.
  \end{remark}

  Our next result concerns the maximal monotonicity of the abstract subdifferential operator. It is not hard to see that the $L$-subdifferential is $L$-monotone~\parencite[Proposition 1.9]{FoundationsOfPallas1997}. However, \textcite[example 3.1]{burachik.ea.08} showed that it is not generally maximal $L$-monotone. Several authors have investigated conditions under which it is~\parencite[among others]{eberhard.ea:2010,eberhard.ea:2011,mohebi.ea:2012,mohebi.ea:2012_2,mohebi.ea:2014}. Here we show that under Assumption~\ref{ass:bounded} the result is also true. Our proof generalises the proof by~\textcite{ivanov.ea:2017}.

  \begin{example}
  Let $X$ be a Banach space, and $f$ any convex function.  Since the domain of $g_a\coloneqq a\|\cdot - x\|$ is $X$ for any $x$ and any $a>0$, then $\ri(\dom(f))\cap \ri(\dom(g_a))$ is nonempty, and it is known that this implies that $\epi f^* + \epi g_a^* = \epi (f+g_a)^*$ (see~\cite{burachik.ea:2004}). From this we infer that $f$ satisfies Assumption~\ref{ass:bounded}.

  \end{example}

  \begin{proposition}\label{prop:zeroinsubdifzero}
    Let $f$ be a proper lsc $L$-convex function that satisfies Assumption~\ref{ass:bounded} and
    \begin{equation}\label{eq:variationalIneq}
      u(x)\geq 0, \forall x\in \dom f, u\in \partial_L f(x).
    \end{equation}
    Then $0\in \partial_L f(0)$.
  \end{proposition}

  \begin{proof}
    For $a>0$, define $g_a\coloneqq a\|\cdot\|$.

    If $p\in \partial_{\mathcal{L}} g_a(x)$ then, $\langle p,0 - x\rangle\leq a\|0\|-a\|x\|$, that is,
    \begin{equation}\label{eq:pgreaterg}
      p(x)\geq g_a(x).
    \end{equation}

    Define $f_a(x)\coloneqq f(x) + g_a(x)$. Then $\partial_{L+\mathcal{L}} f_a(x) = \partial_{L} f(x) + \partial_{\mathcal{L}} g_a(x)$. Let $p\in \partial_{L+\mathcal{L}} f_a(x)$. Then there exist $p_1\in \partial_{L} f(x)$ and $p_2\in \partial_{\mathcal{L}} g_a(x)$ such that $p=p_1+p_2$. Then, we have that
    \begin{equation}
      p(x) = p_1(x) + p_2(x)\geq 0 + g_a(x) = g_a(x).
    \end{equation}

    Let $x_n$ be a minimising sequence of $f_a$, that is $f_a(x_n) < f_a(x) + \varepsilon_n$, $\forall x\in X$ for some $\varepsilon_n\to 0$. Applying Theorem~\ref{thm:bronstedrockafellar}, we find that there exists $y_n\in  B(x_n,\sqrt{\varepsilon_n})$ and $p_n\in \partial_{L+\mathcal{L}} f_a(y_n)$ such that $\exists l\in B(0,\sqrt{\varepsilon_n}): p_n-0 = \langle l,\cdot \rangle$.

    Therefore we have $g_a(y_n)\to 0$, and by continuity of $g_a$, $x_n\to 0$. So, $0$ is the global minimiser of $f_a$.

    In other words, $f_a(x)\geq f_a(0)$. Equivalently, $f(x)\geq f(0) - g_a(x) = f(0)-\|x\|$. Since $a$ was arbitrary, we conclude that $f(x)\geq f(0)$ and therefore $0\in \partial_L f(0)$.
  \end{proof}

  \begin{theorem}
    Let $f$ be an $L$-convex function satisfying Assumption~\ref{ass:bounded}. Then $\partial_L f$ is a maximal $L$-monotone operator. 
  \end{theorem}
  
  \begin{proof}
    First we show that $\partial_L f$ is $L$-monotone. Let $x,y \in X$, and $u\in \partial_L f(x)$ and $v\in \partial_L f(y)$. Then by definition of the abstract subdifferential,
    $$ f(x) \geq f(y) + u(y) - u(x) \geq f(x) + v(x) - v(y) + u(y) - u(x). $$
    Therefore
    $$v(y) - v(x) + u(x) - u(y) \geq 0.$$

  Now suppose that $(y,v) \in X\times L$ is such that \begin{equation}\label{eq:monotone}u(x) - u(y) + v(y) -v(x)\geq 0\end{equation} for all $x\in X$, $u\in \partial_Lf(x)$.

  Define $\bar{f}(x)\coloneqq f(x+y) - v(x+y)$ and the following sets:
  \begin{itemize}
    \item $L^v\coloneqq L-v = \{u-v: u\in L\}$
    \item $L^{v,y}\coloneqq \{x\mapsto l(x+y): l\in L^v\} = \{x\mapsto u(x+y) - v(x+y): u\in L\}$
    \item $L^{v,y}_0 \coloneqq \{l-l(0): l\in L^{v,y}\} = \{x\mapsto u(x+y) - v(x+y) + v(y)-u(y): u\in L\}$.
  \end{itemize}
  Then the function $\bar{f}$ is $L^{v,y}$-convex, and more generally $H_{L^{v,y}_0}$-convex. By Remark~\ref{rmk:diffconvexlinear}, $\partial_{L^{v,y}_0}\bar{f}(x) = \{x\mapsto u(x+y)-u(y)-v(x+y)+v(y): u\in \partial_{L}f(x+y)\}$. It follows from Equation~\eqref{eq:monotone} that for any $l\in \partial_{L^{v,y}_0}(x)$, $l(x)\geq 0$.

  Hence, Proposition~\ref{prop:zeroinsubdifzero} implies that $0\in \partial_{L^{v,y}_0} \bar{f}(0)$. That is, $\exists u\in \partial_L f(0+y)$ such that $u-v-u(y)+v(y) = 0$, that is, $u = v + u(y) -v(y)$.

  Therefore, we have that for any $x\in X$, $$f(x)\geq f(y) + (v(x) + u(y) - v(y)) - (v(y) + u(y) - v(y)) = f(y) + v(x) - v(y).$$ Therefore, $v\in \partial_L f(y)$ which concludes our proof.
  \end{proof}

  \begin{remark}
    \textcite[proposition 1.1.11]{FoundationsOfPallas1997} proved that every maximal cyclic monotone operator is the subdifferential of some function $f$.
  \end{remark}

\section{Conclusions and future research directions}\label{sec:conclusions}

In this paper we obtain summation and composition rules for subdifferential calculus, and prove that under some reasonable conditions the subdifferential is a maximal abstract monotone operator. We also highlighted one of the challenges in the extension of abstract convexity and developing numerical methods: the fact that it is not always possible to separate two abstract convex disjoint sets. This leads us to Open Problem~\ref{prob:Minkowski}.

  \begin{problem}\label{prob:Minkovski}
    What conditions over the set $L$ are sufficient for a Minkovski-like separation result?
  \end{problem}

It is natural to assume that, similar to classical convex analysis, the optimality conditions are formulated in terms of maximal deviations (Proposition~\ref{prop:maximumrule}). This is where Open Problem~\ref{prob:Minkovski} appeared. This problem potentially leads to one more subdifferential calculus rule (maximum).

Finally, this paper identifies one more potential application for abstract convexity: discrete and integer optimisation.

  \printbibliography
\end{document}